\documentclass{amsart}
\usepackage{amsfonts,amssymb,amsmath,amsthm}
\usepackage{url}
\usepackage{enumerate}
\usepackage[bookmarksnumbered, colorlinks, plainpages]{hyperref}
\usepackage{cases}
\newtheorem{theorem}{Theorem}[section]
\newtheorem{lemma}[theorem]{Lemma}
\newtheorem{proposition}[theorem]{Proposition}
\newtheorem{corollary}[theorem]{Corollary}
\theoremstyle{definition}
\newtheorem{definition}[theorem]{Definition}
\newtheorem{remark}[theorem]{Remark}
\newtheorem{example}[theorem]{Example}
\numberwithin{equation}{section}

\begin{document}
\title[Cauchy-Riemann structure and integrability conditions of $F$-structures\ldots]{Exploring the Cauchy-Riemann structure and integrability conditions of $F$-structures  satisfying  $\alpha F^{K+1} +\beta F^{K} + F=0$ }
\author{Abderrahim ZAGANE}
\address{RELIZANE University, Faculty of Science and Technology, Department of 
Mathematics, 48000, RELIZANE-ALGERIA}
\email{Zaganeabr2018@gmail.com}

\begin{abstract}
This work introduces a new class of $F$-structures satisfying $\alpha F^{K+1} +\beta F^{K} + F=0$, where $K$ is a positive integer, $K\geq3$, and $\alpha, \beta$ are real or complex numbers. We investigate the Cauchy-Riemann structure and its link with the $F$-structure. We also address the integrability of this structure, including partial and complete integrability. Finally, we present several examples of the $F$-structure.

	
\end{abstract}

\maketitle

\section{\protect\bigskip Introduction}\label{Sec1}

The $F$-structure satisfying the equation $F^{3} + F = 0$, where $F$ is a nonzero tensor field of type $(1, 1)$ on a differentiable manifold, has been extensively studied by various researchers, including Yano \cite{Yan1,Yan2}, Ishihara and Yano\cite{I.Y}, and Nakagawa \cite{Nak}. Likewise for the $F$-structure that satisfies $F^{3} - F = 0$, it has also been explored by Singh and Vohra \cite{S.V}, Matsumoto \cite{Mat}, and Baik \cite{Bai}. They have been discussed by various authors in different ways, see \cite{Das1, Das2, D.N.N, G.Y, U.G}. In the previous articles, various aspects of the $F$-structure have been explored, including integrability conditions, $CR$-structures, parallelism of distributions, and submanifolds within the $F$-structure manifold. Recent studies have been conducted on the F-structure, as discussed in several papers \cite{Kha, S.P.K, S.G1, S.G2, Sin1}. Due to this reason, the field of structures on a differentiable manifold continues to be a fascinating area of research in differential geometry, attracting significant attention even today.

The primary objective of this paper is to investigate the $F$-structure that satisfies the equation $\alpha F^{K+1} +\beta F^{K} + F=0$, where $K$ is a positive integer, $K\geq3$, and $\alpha, \beta$ are real or complex numbers.. After providing an introduction, we delve into section \ref{Sec2} to explore the fundamental properties of operators $l$ and $m$ as defined by the $F$-structure. In section \ref{Sec3}, we discuss the properties of the Nijenhuis tensor of $F$, $l$, and $m$. In Section \ref{Sec4}, we analyze the Cauchy-Riemann structure and explore the relationship between the $F$ structure. In section \ref{Sec5}, we present the necessary and sufficient conditions for integrability of the distributions induced by operators $l$, and $m$. In section \ref{Sec6}, we explore the Partial integrability and complete integrability conditions of the $F$-structure. In Section \ref{Sec7}, various previously researched structures are presented and considered as special cases of the structure $F$, demonstrating that the $F$-structure  is a generalization of them. In the last section, we present a few examples of the $F$-structure.  

Let $M^{n}$ be an $n$-dimensional manifold. A distribution $D$ of dimension $k$ on $M$ is a subbundle of $TM$ such that, for all point $x$ of $M$, $D_{x}$ is a $k$-dimensional subspace of $T_{x}M$. A vector field $X$ on $M$ is said to belong (tangent) to $D$ if $X_{x} \in D_{x}$ for all $x \in M$. The set of vector field belong to $D$ is also denoted by $D$. $D$ is said to be involutive if $[X, Y]$ belongs to $D$ for every vector fields $X, Y$ belonging to $D$ i.e. $[X, Y]\in D$, for every vector fields $X, Y \in D$. A submanifold $N$ of $M$ is called an integral manifold of $D$, if $T_{x}N = D_{x}$ for any point $x \in N$. We say the distribution $D$ is integrable if through each point of $M$ there exists an integral manifold of $D$. We need the classical theorem of Frobenius, which we formulate as follows (\cite[p.197]{B.C}). A distribution is integrable if and only if it is involutive (see \cite{K.N,D.R} for more details).\quad

\section{New classes of $F$-structure satisfying $\alpha F^{K+1} + \beta F^{K} + F = 0$}\label{Sec2}

Let $M^{n}$ be an $n$-dimensional manifold and $F$ be a nonzero $(1, 1)$-tensor field on $M$ 
of rank $rank(F)=r$ satisfying the polynomial equation:
\begin{equation}\label{eq_A}
\alpha F^{K+1} + \beta F^{K} + F=0, 
\end{equation}
where $K$ is a positive integer, $K\geq3$ and $\alpha, \beta$ are real or complex numbers. Such a structure on $M$ is called an $F$-structure of rank $r$ and of degree $K$. If the rank of $F$, $rank(F)=r=$ constant, then $M$ is called an $F$-structure manifold of degree $K$.

We define two operators $l$ and $m$ on $M$ respectively by 
\begin{eqnarray}\label{eq_B}
l&=&-(\alpha F^{K} + \beta F^{K-1}), \\\label{eq_C}
m&=& I +\alpha F^{K} + \beta F^{K-1}, 
\end{eqnarray}
where $I$ denotes the identity operator on $M$, then we get

\begin{lemma}\label{lem_1} Given an $F$-structure manifold $M$, then  we have
\begin{eqnarray}\label{eq_D}
l+m &=& I,\\\label{eq_E}
l^{2}&=&l,\\\label{eq_F}
m^{2}&=&m,\\\label{eq_G}
lF &=& Fl = F,\\\label{eq_H} 
mF &=& Fm = 0,\\ \label{eq_I}
lm &=& ml = 0.
\end{eqnarray}
\end{lemma}

\begin{proof}
	
$i)$ Combining \eqref{eq_B} and \eqref{eq_C} we get \eqref{eq_D}.
\begin{eqnarray*}	
ii)\;l^{2} &=&(\alpha F^{K}+ \beta F^{K-1})^{2}\\
&=& (\alpha F^{K+1}+ b F^{K})(\alpha F^{K-1}+ \beta F^{K-2})\qquad\qquad\qquad\qquad\qquad\qquad\qquad\qquad\\
&=&-F(\alpha F^{K-1}+ \beta F^{K-2})\\
&=&-(\alpha F^{K} + \beta F^{K-1})=l.
\end{eqnarray*}
\begin{eqnarray*}	
iii)\;m^{2} = (I-l)^{2}=I-2l+l^{2}=I-2l+l=I-l=m.\qquad\qquad\qquad\qquad\qquad\qquad
\end{eqnarray*}
\begin{eqnarray*}	
iv)\; lF &=& -(\alpha F^{K}+ \beta F^{K-1})F\qquad\qquad\qquad\qquad\qquad\qquad\qquad\qquad\qquad\qquad\qquad\\
&=&-(\alpha F^{K+1}+ \beta F^{K})\\
&=&F.
\end{eqnarray*}
\begin{eqnarray*}	
v)\;mF =(I-l)F=F-lF=0.\qquad\qquad\qquad\qquad\qquad\qquad\qquad\qquad\qquad\qquad\qquad\qquad\qquad
\end{eqnarray*}
\begin{eqnarray*}	
vi)\; lm = l(I-l)=l-l^{2}=l-l=0.\qquad\qquad\qquad\qquad\qquad\qquad\qquad\qquad\qquad
\end{eqnarray*}
\end{proof}
If, we put $\widehat{F}=(\alpha F^{K}+ \beta F^{K-1})^{1/2}$, then, we have
\begin{lemma}\label{lem_2}	
Given an $F$-structure manifold $M$, then  we have
\begin{eqnarray}\label{eq_J}
l\widehat{F} &=& \widehat{F}l = \widehat{F},\\\label{eq_K}
m\widehat{F} &=& \widehat{F}m =0.
\end{eqnarray}
\end{lemma}
\begin{proof}In consequence of \eqref{eq_G} and \eqref{eq_H}, we get \eqref{eq_J}
and \eqref{eq_K}.
\end{proof}

\begin{proposition}\label{pro_0}	
Given an $F$-structure manifold $M$, the following identities hold
\begin{eqnarray}\label{eq_L}
\mathrm{Im}\,l_{x} &=& \ker m_{x},\\\label{eq_M}
\mathrm{Im}\,m_{x} &=& \ker l_{x},\\\label{eq_N1}
\mathrm{Im}\,l_{x} &=& \mathrm{Im}\,F_{x},\\\label{eq_N2}
\ker l_{x} &=& \ker F_{x},\\\label{eq_O}
T_{x}M &=& \mathrm{Im}\,l_{x} \oplus \mathrm{Im}\,m_{x},\\\label{eq_Q}
\dim (\mathrm{Im}\,l_{x}) &=&r\,, \;\dim (\mathrm{Im}\,m_{x})\;=\;n-r.\\\label{eq_O1}
\ker l_{x} &=& \ker \widehat{F}_{x},
\end{eqnarray}
for all $x \in M$ and any vector fields $X$, $Y$ on $M$, where $\widehat{F}=(\alpha F^{K}+ \beta F^{K-1})^{1/2}$.
\end{proposition}

\begin{proof}
for all $x \in M$ and $X\in T_{x}M$,
\item $(i)$ If $X\in \mathrm{Im}\,l_{x} $, There is  $Z\in T_{x}M,\; X= lZ$, using \eqref{eq_I}, we have\\
$mX=mlZ=0$, then $X\in \ker m_{x}$.\\
Conversely, If $X\in \ker m_{x} $, so $mX=0$, using \eqref{eq_D}, we have $lX=X$, and from it
$X\in \mathrm{Im}\,l_{x} $.\; Therefore $\mathrm{Im}\,l_{x}= \ker m_{x}$.
\item $(ii)$ The formula \eqref{eq_M} is obtained by a proof similar to that of the formula \eqref{eq_L}.
\item $(iii)$ If $X\in \mathrm{Im}\,l_{x}$, There is  $Z\in T_{x}M,\; X= lZ$, using (\ref{eq_B}), we have\\
$X=-(\alpha F^{K}+ \beta F^{K-1})Z=FY$, \\
where $Y=-(\alpha F^{K-1}+ \beta F^{K-2})Z \in T_{x}M$, then $X\in \mathrm{Im}\,F_{x}$.\\
Conversely, If $X\in \mathrm{Im}\,F_{x}$, There is  $Z\in T_{x}M,\; X= FZ$, using \eqref{eq_G}, we have
$X=lFZ=lY$, where $Y=FZ \in T_{x}M$, then $X\in \mathrm{Im}\,l_{x}$.
\item $(vi)$ The formula \eqref{eq_N2} is obtained by a proof similar to that of the formula \eqref{eq_N1}.
\item $(v)$ By applying the well-known rank theorem in linear algebra on $T_{x}M$, we 
find $T_{x}M = \mathrm{Im}\,l_{x}\oplus \ker l_{x}$, using  \eqref{eq_M}, we get $T_{x}M= \mathrm{Im}\,l_{x} \oplus \mathrm{Im}\,m_{x}$.
\item $(vi)$ By \eqref{eq_N1} and \eqref{eq_O}, we find
$\dim (\mathrm{Im}\,l_{x})=\dim (\mathrm{Im}\,F_{x})= rank(F)=r$ and $\dim (\mathrm{Im}\,m_{x})=n-r$. 
\item $(vii)$ The formula \eqref{eq_O1} is obtained by a proof similar to that of the formula \eqref{eq_N1}.
\end{proof}	

Thus the operators $l$ and $m$ acting in the tangent space at each point of $M$ are 
therefore complementary projection operators and  there exist two complementary distributions 
$D_{l}= \mathrm{Im}\,l$ and $D_{m}=\mathrm{Im}\ m$ corresponding to the projection operators $l$ and $m$ 
respectively. The dimensions of $D_{l}$ and $D_{m}$ are $r$ and $n-r$ respectively.\qquad

From \eqref{eq_J}, we find 
\begin{eqnarray*}
\widehat{F}^{2}l=\widehat{F}^{2}=\alpha F^{K}+ \beta F^{K-1}=-l,
\end{eqnarray*}

it is clear that $\widehat{F}=(\alpha F^{K}+ \beta F^{K-1})^{1/2}$ acts on $D_{l}$ as an almost complex structure and on $D_{m}$ as a null operator. Hence $r$ must be even, say $r = 2k$. 

\section{Nijenhuis tensor}\label{Sec3}
The Nijenhuis tensor $N_{F}$ of $F$ is expressed as follows
\begin{eqnarray}\label{Nijenhuis}
N_{F}(X, Y) = [FX, FY] - F[FX, Y] -F[X, FY] + F^{2}[X, Y],
\end{eqnarray}
for any vector fields $X$ and $Y$ on $M$.

The integrability of $F$-structure is equivalent to the vanishing of the Nijenhuis tensor \cite{D.R,I.Y,Y.K}.

The Nijenhuis tensor $N_{F}$ satisfies the following relations:
\begin{eqnarray}\label{eq_00}
N_{F}(mX, mY) &=& F^{2}[mX, mY],\\\label{eq_01}
lN_{F}(mX, mY) &=& F^{2}[mX, mY],\\\label{eq_02}
mN_{F}(X, Y) &=& m[FX, FY],\\\label{eq_03}
mN_{F}(FX, FY) &=& m[F^{2}X, F^{2}Y],\\\label{eq_04}
mN_{F}(lX, lY) &=& m[FX, FY],
\end{eqnarray}
for any vector fields $X$ and $Y$ on $M$.

\begin{proposition}\label{pro_00}	
Given an $F$-structure manifold $M$, we have the following equivalences
\begin{eqnarray}\label{eq_12.0}
mN_{F}(X, Y)=0&\Leftrightarrow&mN_{F}(FX, FY)=0 \;\Leftrightarrow\;mN_{F}(lX, lY)=0,
\end{eqnarray}
for any vector fields $X$ and $Y$  on $M$.
\end{proposition}

\begin{proof}
The proof follows from \eqref{eq_02}, \eqref{eq_03} and \eqref{eq_04}.
\end{proof}

\begin{proposition}\label{pro_01}	
Given an $F$-structure manifold $M$, the following conditions are equivalent: 
\begin{eqnarray}\label{eq_12.1}
N_{F}(FX, FY)=0 \Leftrightarrow N_{F}(lX, lY)=0,
\end{eqnarray}
for any vector fields $X$ and $Y$  on $M$.
\end{proposition}

\begin{proof}
$(i)$ Assume that $N_{F}(FX, FY)=0$, we replace $X$ with $-(\alpha F^{K-1}+\beta F^{K-2})X$ and $Y$ with $-(\alpha F^{K-1}+\beta F^{K-2})Y$, respectively, we obtain $N_{F}(lX, lY)=0$.
\item $(ii)$ Conversely, assume that $N_{F}(lX, lY)=0$, we replace $X, Y$ with $FX$, $FY$, respectively, we obtain $N_{F}(FX, FY)=0$.
\end{proof}

\begin{proposition}\label{pro_02} Given an $F$-structure manifold $M$. If the $F$-structure  is 
integrable, then we have
\begin{eqnarray*}
(i)\; F[X, Y]&=& (\alpha F^{K-1}+ \beta F^{K-2})[FX, FY]+l\big([F X, Y] + [X, FY]\big),\\
(ii)\; [FX, FY]&=&l [FX, FY],\\
(iii)\; m[FX, FY]&=&0,
\end{eqnarray*}
for any vector fields $X$ and $Y$  on $M$.
\end{proposition}

\begin{proof}
$(i)$\; Since $N_{F}(X,Y)=0$ we obtain
\begin{eqnarray*}
[FX, FY]+F^{2}[X, Y]=F\big([FX, Y]+[X, FY]\big)
\end{eqnarray*}
Operating on it by $-(\alpha F^{K-1}+ \beta F^{K-2})$ we get	
\begin{eqnarray*}
-(\alpha F^{K-1}+ \beta F^{K-2})\big([FX, FY] + F^{2}[X, Y]\big)=\qquad\qquad\qquad\qquad\qquad\qquad\\-(\alpha F^{K}+ \beta F^{K-1})\big([FX, Y]+ [X, F Y]\big).
\end{eqnarray*}
Using \eqref{eq_B} and \eqref{eq_G}, we find
\begin{eqnarray*}
-(\alpha F^{K-1}+ \beta F^{K-2})[FX, FY] + F[X, Y]=l\big([FX, Y]+ [X, F Y]\big).
\end{eqnarray*}
$(ii)$\; From \eqref{eq_G}, we find
\begin{eqnarray}\label{eq_1.1}
N_{F}(X, Y)-lN_{F}(X, Y) = [FX, FY] - l[FX, FY],
\end{eqnarray}
since $N_{F}(X,Y)=0$ we obtain, $[FX, FY]=l[FX, FY]$.\\
$(iii)$\; Using \eqref{eq_D}, we find $m[FX, FY]=0$.
\end{proof}

Let $N_{l}$  and $N_{m}$  denote the Nijenhuis tensors corresponding to the operators $l$ and 
$m$ respectively, then
\begin{eqnarray*}
N_{l}(X, Y) &=& [lX, lY] - l[lX, Y] -l[X, lY] + l[X, Y],\\
N_{m}(X, Y)	&=& [mX, mY] - m[mX, Y] -m[X, mY] + m[X, Y],
\end{eqnarray*}
for any vector fields $X$ and $Y$  on $M$.

\begin{proposition}\label{pro_1}	
Given an $F$-structure manifold $M$, then  we have
\begin{eqnarray}\label{eq_R}
N_{l}(X, Y) = N_{m}(X, Y) =m[lX, lY] +l[mX, mY],
\end{eqnarray}
for any vector fields $X$ and $Y$ on $M$.
\end{proposition}

\begin{proof}
Using (\ref{eq_D}), we have, $lX+mX=X$, then
\begin{eqnarray*}
N_{l}(X, Y) &=& [lX, lY] - l[l X, lY+mY] -l[lX+mX, lY]\\
&& + l[lX+mX, lY+mY],\\
&=& [lX, lY] - l[lX, lY] -l[lX, mY] - l[lX, lY] -l[mX, lY]\\
&& +l[lX, lY]+l[lX, mY]+l[mX, lY]+l[mX, mY]\\
&=& [lX, lY]- l[lX, lY]+l[mX, mY]\\
&=& m[lX, lY] +l[mX, mY].
\end{eqnarray*}
\begin{eqnarray*}
N_{m}(X, Y) &=& [mX, mY] - m[m X, lY+mY] -m[lX+mX, mY]\\
&& + m[lX+mX, lY+mY],\\
&=& [mX, mY] - m[mX, lY] -m[mX, mY] - m[lX, mY] -m[mX, mY]\\
&& +m[lX, lY]+m[lX, mY]+m[mX, lY]+m[mX, mY]\\
&=& [mX, mY]- m[mX, mY]+m[lX, lY]\\
&=& m[lX, lY] +l[mX, mY].
\end{eqnarray*} 
\end{proof}

By virtue of Proposition \ref{pro_1}, we get the following proposition. 

\begin{proposition}\label{pro_1.1}
Given an $F$-structure manifold $M$, the both operators $l$ and $m$ are integrable if and only if 	
\begin{eqnarray*}
N_{l}(X, Y)=0, 
\end{eqnarray*}
or
\begin{eqnarray*}
m[lX, lY] =-l[mX, mY], 
\end{eqnarray*}
for any vector fields $X$ and $Y$ on $M$.
\end{proposition}

\begin{proposition}\label{pro_2}	
Given an $F$-structure manifold $M$, the following identities hold
\begin{eqnarray}
\label{eq_S} N_{l}(lX, lY)&=& m[lX, lY],\\
\label{eq_T} N_{l}(mX, mY)&=& l[mX, mY],\\\
\label{eq_W} N_{l}(X, Y)&=&N_{l}(lX, lY) +N_{l}(mX, mY),\\
\label{eq_W1} mN_{F}(X, Y)&=&N_{l}(FX, FY),\\
N_{l}(lX, mY)&=&0,\nonumber\\
N_{l}(mX, lY)&=&0,\nonumber
\end{eqnarray}
for any vector fields $X$ and $Y$ on $M$.
\end{proposition}

\begin{proof}
By virtue of \eqref{eq_E}, \eqref{eq_F}, \eqref{eq_I} and \eqref{eq_R} we get
\begin{eqnarray*}
(i) && N_{l}(lX, lY)= m[l^{2}X, l^{2}Y] +l[mlX, mlY]=m[lX, lY],\\
(ii) && N_{l}(mX, mY)=m[lmX, lmY] +l[m^{2}X, m^{2}Y]= l[mX, mY],\\
(iii) && N_{l}(lX, mY)=m[l^{2}X, lmY] +l[mlX, m^{2}Y]=0,\\
(iv) && N_{l}(mX, lY)=m[lmX, l^{2}Y] +l[m^{2}X, mlY]=0.
\end{eqnarray*}
$(v)$ By virtue of \eqref{eq_R}, \eqref{eq_S} and \eqref{eq_T} we get \eqref{eq_W}.\\
$(vi)$ In \eqref{eq_S}, replacing $X, Y$  with  
$FX, FY$, respectively, we find 
$$N_{l}(FX, FY)= m[FX, FY].$$ 
By (\ref{eq_02}), we obtain \eqref{eq_W1}.
\end{proof}

\begin{proposition}\label{pro_6}	
Given an $F$-structure manifold $M$, the following identity hold
\begin{eqnarray*}
N_{F}(mX, mY)=F^{2}N_{l}(mX, mY),
\end{eqnarray*}
for any vector fields $X$ and $Y$  on $M$.
\end{proposition}

\begin{proof}
By virtue of \eqref{eq_T}, we have $N_{l}(mX, mY)= l[mX, mY]$, operating on it by $F^{2}$, 
we find $F^{2}N_{l}(mX, mY)= F^{2}[mX, mY]$.\\
On the other hand by \eqref{eq_00}, we have  $N_{F}(mX, mY) = F^{2}[mX, mY]$.\\ 
Hence $N_{F}(mX, mY)=F^{2}N_{l}(mX, mY)$.
\end{proof}

\section{Cauchy-Riemann structure}\label{Sec4}
Let $T^{\mathbb{C}}M$ denotes the complexified tangent bundle of differentiable manifold $M$ defined by
$$T^{\mathbb{C}}M = \left\lbrace X + j Y :\; X,Y \in TM \right\rbrace =TM\otimes_{\mathbb{R}}\mathbb{C},$$
where $j$ is the imaginary unit. 

A $CR$-structure on $M$ is a complex subbundle $H$ of $T^{\mathbb{C}}M$ such that $H \cap \overline{H}=\{0\}$ and $H$ is involutive, where $\overline{H}$ denotes the complex conjugate of $H$. In this case, we say $M$ is a $CR$-manifold.

Let $F$-structure on $M$ of rank $r = 2k$ satisfying the equation \eqref{eq_A}. We define complex subbundle $H$ of $T^{\mathbb{C}}M$ by
\begin{eqnarray}\label{eq_3.1}
H=\left\lbrace X-j\widehat{F}X,\; X\in D_{l}\right\rbrace.
\end{eqnarray}
where $\widehat{F}=(\alpha F^{K}+ \beta F^{K-1})^{1/2}$. Then, we have 
\begin{eqnarray}\label{eq_3.2}
Real(H)=D_{l}\;\; \textit{and}\;\; H \cap \overline{H}=\{0\}.	
\end{eqnarray} 
Indeed
\begin{eqnarray*}
Z\in H \cap \overline{H}&\Rightarrow& Z=X-j\widehat{F}X=X+j\widehat{F}X,\; X\in D_{l}\\
&\Rightarrow& \widehat{F}X=0\\
&\Rightarrow& Z=X\in \ker \widehat{F},
\end{eqnarray*}
from, \eqref{eq_M}, \eqref{eq_O} and \eqref{eq_O1}, we have $Z\in \ker \widehat{F}=\ker l=D_{m}$\\
$Z\in  D_{l}\cap D_{m}=\{0\}$.

\begin{lemma}\label{lem_4}	
Given an $F$-structure manifold $M$, the following identity hold
\begin{eqnarray}\label{eq_3.3}
[P,Q]=[X,Y]-[\widehat{F}X,\widehat{F}Y]-j([\widehat{F}X,Y]+[X,\widehat{F}Y]),
\end{eqnarray}
for any $P=X-j\widehat{F}X, Q=Y-j\widehat{F}Y \in H$, where $X,Y \in D_{l}$.
\end{lemma} 

\begin{proof}
\begin{eqnarray*}
[P,Q]&=&[X-j\widehat{F}X,Y-j\widehat{F}Y]\\
&=&[X,Y]+[X,-j\widehat{F}Y]+[-j\widehat{F}X,Y]+[-j\widehat{F}X,-j\widehat{F}Y])\qquad\qquad\qquad\qquad\qquad\qquad\qquad\\
&=&[X,Y]-[\widehat{F}X,\widehat{F}Y]-j([\widehat{F}X,Y]+[X,\widehat{F}Y]).
\end{eqnarray*}	
\end{proof}	

\begin{lemma}\label{lem_4}	
Given an $F$-structure manifold $M$, the following identity hold
\begin{eqnarray}\label{eq_3.4}
l([\widehat{F}X, Y] + [X, \widehat{F}Y]) = [\widehat{F}X, Y] + [X, \widehat{F}Y],\\\label{eq_3.5}	
l[\widehat{F}X, \widehat{F}Y]=[\widehat{F}X, \widehat{F}Y],\qquad\qquad\qquad\qquad
\end{eqnarray}
for any $X,Y \in D_{l}$.
\end{lemma} 

\begin{proof}  
\begin{eqnarray*}
l([\widehat{F}X, Y] + [X, \widehat{F}Y])&=& l\big(\widehat{F}X.Y - Y.\widehat{F}X + X.\widehat{F}Y - \widehat{F}Y.X\big)\\
&=& l\widehat{F}X.Y - lY.\widehat{F}X + lX.\widehat{F}Y - l\widehat{F}Y.X
\end{eqnarray*}
as $X, Y\in D_{l}$, we have $lX=X, lY=Y$ and using \eqref{eq_J}, we get	
\begin{eqnarray*}
l([\widehat{F}X, Y] + [X, \widehat{F}Y])&=& \widehat{F}X.Y - Y.\widehat{F}X + X.\widehat{F}Y - \widehat{F}Y.X\\
&=& [\widehat{F}X, Y] + [X, \widehat{F}Y].
\end{eqnarray*}	
The formula \eqref{eq_3.5} is obtained by a similar calculation.
\end{proof}	

\begin{theorem}\label{th_0.0} Given an $F$-structure manifold $M$. If $\widehat{F}$ is integrable, then the complex subbundle $H$ defined by \eqref{eq_3.1} is a $CR$-structure on $M$.
\end{theorem}

\begin{proof} From \eqref{eq_3.2}, we have $Real(H)=D_{l}$  and $H \cap \overline{H}=\{0\}$. It remains to show that $H$ is involutive, let $P=X-j\widehat{F}X,\;Q=Y-j\widehat{F}Y \in H$, such that $X,Y \in D_{l}$. Using using \eqref{eq_3.3}, \eqref{eq_3.4} and \eqref{eq_3.5}, we get
\begin{eqnarray*}
[P,Q]&=&[X,Y]-[\widehat{F}X,\widehat{F}Y]-j([\widehat{F}X,Y]+[X,\widehat{F}Y])
\end{eqnarray*}
Since $\widehat{F}$ is integrable, then $N_{\widehat{F}}(X,Y)=0$ i.e.
\begin{eqnarray*}
[\widehat{F}X, \widehat{F}Y]-[X, Y]=\widehat{F}\big([\widehat{F}X, Y]+[X, \widehat{F}Y]\big),
\end{eqnarray*}	
operating on it by $-\widehat{F}$ we get
\begin{eqnarray*}
-\widehat{F}\big([\widehat{F}X, \widehat{F}Y]-[X, Y]\big)=l[\widehat{F}X, Y]+[X, \widehat{F}Y],
\end{eqnarray*}	
hence,
\begin{eqnarray*}
\widehat{F}\big([X, Y]-[\widehat{F}X, \widehat{F}Y]\big)=[\widehat{F}X, Y]+[X, \widehat{F}Y],
\end{eqnarray*}
from that we find,	 
\begin{eqnarray*}
[P,Q]&=&[X,Y]-[\widehat{F}X,\widehat{F}Y]-j\widehat{F}\big([X, Y]-[\widehat{F}X, \widehat{F}Y]\big)\in H.
\end{eqnarray*}	
\end{proof}

\section{Integrability conditions of distributions induced of $F$-structure}\label{Sec5}

\begin{theorem}\label{th_0} Given an $F$-structure manifold $M$. The distribution $D_{l}$ is  
integrable if and only if
\begin{eqnarray}\label{eq_X}
N_{l}(lX, lY)=0,
\end{eqnarray}
or
\begin{eqnarray*}
m[lX,  lY]=0,
\end{eqnarray*}
for any vector fields $X$ and $Y$ on $M$.
\end{theorem}

\begin{proof} The  distribution $D_{l}$ is integrable if and only if for any vector fields 
$X$ and $Y$ on $M$ we have 
$$[lX,  lY] \in D_{l}.$$
By virtue of \eqref{eq_L} and \eqref{eq_S} we get,
\begin{eqnarray*}
[lX,  lY] \in D_{l}&\Leftrightarrow&  m[lX,  lY]=0 \Leftrightarrow N_{l}(lX, lY)=0.
\end{eqnarray*}
\end{proof}

\begin{theorem}\label{th_1} Given an $F$-structure manifold $M$. The distribution $D_{m}$ 
is integrable if and only if
\begin{eqnarray}\label{eq_Y}
N_{l}(mX, mY)=0,
\end{eqnarray}
or
\begin{eqnarray*}
l[mX,  mY]=0,
\end{eqnarray*}
for any vector fields $X$ and $Y$ on $M$.
\end{theorem}

\begin{proof} The  distribution $D_{m}$ is integrable if and only if for any vector fields 
$X$ and $Y$ on $M$ we have 
$$[mX,  mY] \in D_{m}.$$
By virtue of \eqref{eq_M} and \eqref{eq_T} we get,
\begin{eqnarray*}
[mX,  mY] \in D_{m}&\Leftrightarrow&  l[mX,  mY]=0 \Leftrightarrow N_{l}(mX, mY)=0.
\end{eqnarray*}
\end{proof}

\begin{theorem}\label{th_2} Given an $F$-structure manifold $M$. The distributions $D_{l}$ 
and $D_{m}$ are both integrable if and only if 
\begin{eqnarray*}
N_{l}(X, Y)=0,
\end{eqnarray*}
or
\begin{eqnarray*}
l[mX,mY ] = -m[lX, lY ],
\end{eqnarray*}
for any vector fields $X$ and $Y$ on $M$.
\end{theorem}

\begin{proof} Suppose that $D_{l}$ and $D_{m}$ are both integrable. It follows from  
\eqref{eq_X} and \eqref{eq_Y} 
$$N_{l}(lX, lY)=0,\;\; N_{l}(mX, mY)=0.$$
By virtue of \eqref{eq_W} we have,
\begin{eqnarray*}
N_{l}(X, Y)&=&N_{l}(lX, lY) + N_{l}(mX, mY)=0.
\end{eqnarray*}
Conversely, assume that $N_{l}(X, Y)=0$. It follows from \eqref{eq_W} that
$$N_{l}(lX, lY)+ N_{l}(mX, mY)=0.$$
Replacing $X,Y$  by  $lX,lY$ $($resp. by  $mX,mY$ $)$, we get
$$N_{l}(lX, lY)=0,(\textit{resp}.\; N_{l}(mX, mY)=0.$$
then, $D_{l}$ and $D_{m}$ are both integrable.
\end{proof}

\begin{theorem}\label{th_4} Let M be an $F$-structure manifold. The distribution $D_{l}$ 
is integrable  if  and  only  if
\begin{eqnarray}\label{eq_1.3}
 N_{F}(X, Y)=l N_{F}(X, Y),
\end{eqnarray}
or
\begin{eqnarray*}
[FX, FY] = l[FX, FY],
\end{eqnarray*}
for any vector fields $X$ and $Y$  on $M$.
\end{theorem}

\begin{proof} Suppose that, $D_{l}$ is integrable, then for any vector fields $X$ and $Y$ 
on $M$ we have 
$$[lX,  lY] \in D_{l}.$$
Using \eqref{eq_D} and \eqref{eq_L}, we get
\begin{eqnarray}\label{eq_1.4}
[lX,  lY] \in D_{l}&\Leftrightarrow&  m[lX,  lY]=0\nonumber\\
&\Leftrightarrow& [lX,  lY]-l[lX,  lY]=0.
\end{eqnarray}
In the last equation we replace $X, Y$  with  $FX, FY$, respectively and using \eqref{eq_G} 
we  obtain
\begin{eqnarray}\label{eq_1.5}
[F X, F Y]-l[F X, F Y]=0.
\end{eqnarray}
Using \eqref{eq_1.1}, we find $N_{F}(X, Y)=lN_{F}(X, Y)$.\\
Conversely, suppose that  $N_{F}(X, Y)=lN_{F}(X, Y)$, then from \eqref{eq_1.1} we  obtain (\ref{eq_1.5}).
Replacing $X$ and $Y$  with  $-(\alpha F^{K-1}+ \beta F^{K-2})X$ and 
$-(\alpha F^{K-1}+ \beta F^{K-2})Y$, respectively and using 
\eqref{eq_B} we  obtain (\ref{eq_1.4}),  which  implies $[lX,  lY] \in D_{l}$ 
i.e $D_{l}$ is integrable.
\end{proof}

\begin{theorem}\label{th_5}	
Given an $F$-structure manifold $M$, the following conditions are equivalent
\begin{eqnarray*}
(i) && D_{l}\;\textit{is\;integrable},\\
(ii) && N_{l}(lX, lY)=0,\\		
(iii) && N_{F}(X, Y)=lN_{F}(X, Y),\\
(iv) && mN_{F}(X, Y)=0,\\
(v) && mN_{F}(FX, FY)=0, \\
(vi) &&mN_{F}(lX, lY)=0,\\
(vii) && N_{l}(FX, FY)=0,
\end{eqnarray*}
for any vector fields $X$ and $Y$  on $M$.
\end{theorem}

\begin{proof}
\item (1) From  Theorem \ref{th_0}, we  have  $(i)\Leftrightarrow (ii)$.
\item (2) From  Theorem \ref{th_4},	we  have  $(i)\Leftrightarrow (iii)$, hence  $(ii)\Leftrightarrow (iii)$
\item (3) From \eqref{eq_D}, we get $(iii)\Leftrightarrow (iv)$ 
\item (4) From \eqref{eq_12.0}, we get   $(iv) \Leftrightarrow (v) \Leftrightarrow (vi)$.
\item (5) From \eqref{eq_W1}, we get   $(iv) \Leftrightarrow (vii)$, hence $(vi) \Leftrightarrow (vii)$.
\item (6) Suppose that $N_{l}(FX, FY)=0$. Comparing it with \eqref{eq_02} and \eqref{eq_W1},
we obtain $m[FX, FY]=0$. In this relation we replace $X$ by  $-(\alpha F^{K-1}+ \beta F^{K-2})X$ and  $Y$ by $-(\alpha F^{K-1}+ \beta F^{K-2})Y$, respectively,  we find
\begin{eqnarray*}
m[lX, lY]=0&\Leftrightarrow& [lX,  lY] \in D_{l}.
\end{eqnarray*}
Hence, $(vii)\Leftrightarrow (i)$.
\end{proof}

\begin{theorem}\label{th_7} Let M be an $F$-structure manifold. The distribution $D_{m}$ is  
integrable  if  and  only  if
\begin{eqnarray}\label{eq_1.7}
N_{F}(mX, mY)=0,
\end{eqnarray}
or
\begin{eqnarray*}
lN_{F}(mX, mY)=0,
\end{eqnarray*}
for any vector fields $X$ and $Y$  on $M$.
\end{theorem}

\begin{proof} The distribution $D_{m}$ is integrable if  and  only  if $[mX,  mY] \in D_{m}$ 
for any vector fields $X$ and $Y$ on $M$.\\
Using \eqref{eq_M}, we get
\begin{eqnarray*}
[mX,  mY] \in D_{m}&\Rightarrow&  l[mX,  mY]=0\\
&\Rightarrow& F^{2}l[mX,  mY]=0\\
&\Rightarrow& F^{2}[mX,  mY]=0.
\end{eqnarray*}
From \eqref{eq_00}, we get $N_{F}(mX, mY)=0$.\\
Conversely, assume  that $N_{F}(mX, mY)=0$, from \eqref{eq_00}, we find $F^{2}[mX,  mY]=0$.
We operate it by $-(\alpha F^{K-2}+ \beta F^{K-3})$, we get $l[mX,  mY]=0$, i.e.
$[mX,  mY] \in D_{m}$, hence the distribution $D_{m}$ is integrable.\\
Using \eqref{eq_01} we get $N_{F}(mX, mY)=0 \Leftrightarrow  lN_{F}(mX, mY)=0$. 
\end{proof}

By virtue of Proposition \ref{pro_6}, Theorem \ref{th_1} and Theorem \ref{th_7}, we get the following
Theorem.
\begin{theorem}\label{th_8} 
Given an $F$-structure manifold $M$, the following conditions are equivalent
\begin{eqnarray*}
(i) && D_{m}\;\textit{is\;integrable},\\
(ii) && N_{l}(mX, mY)=0,\\		
(iii) &&  N_{F}(mX, mY)=0,\\
(iv) &&  l N_{F}(mX, mY)=0,
\end{eqnarray*}
for any vector fields $X$ and $Y$  on $M$.
\end{theorem}

From Theorem \ref{th_4} and Theorem \ref{th_7} we deduce:

\begin{corollary}
Given an $F$-structure manifold $M$. If $F$ is an integrable structure, then both 
distributions $D_{l}$ and $D_{m}$ are integrable.
\end{corollary}

\begin{remark}
Given an $F$-structure manifold $M$. If both distributions $D_{l}$ and $D_{m}$ are integrable, $F$ not necessary integrable. see (Example \ref{ex_4}).
\end{remark}

\begin{theorem}\label{th_9} Given an $F$-structure manifold $M$. The distributions $D_{l}$ 
and $D_{m}$ are both integrable if and only if 
\begin{eqnarray}\label{eq_2.0}
N_{F}(X, Y)=lN_{F}(lX, lY)+N_{F}(lX, mY)+N_{F}(mX, lY),
\end{eqnarray}
for any vector fields $X$ and $Y$  on $M$.
\end{theorem}

\begin{proof}
$i)$ Suppose that  $l$ and $m$  are both integrable. Using \eqref{eq_D}, we get 
\begin{eqnarray}\label{eq_2.1}
N_{F}(X, Y)&=&N_{F}(lX+mX, lY+mY)\notag\\
&=&N_{F}(lX, lY)+N_{F}(lX, mY)+N_{F}(mX, lY)+N_{F}(mX, mY).
\end{eqnarray}
Then from \eqref{eq_1.3} and \ref{eq_1.7}, we have 
\begin{eqnarray*}
N_{F}(X, Y)=l N_{F}(X, Y)\; \textit{and}\; N_{F}(mX, mY)=0.
\end{eqnarray*}
By virtue of (\ref{eq_2.1}), we get (\ref{eq_2.0}).\\
$ii)$ Conversely, assume that (\ref{eq_2.0}) is satisfied. Using (\ref{eq_2.1}), we find
\begin{eqnarray*}
lN_{F}(lX, lY)&=&N_{F}(lX, lY)+N_{F}(mX, mY).
\end{eqnarray*}
In this relation we replace $X$ and $Y$ with $mX$ and $mY$ respectively, we obtain
$N_{F}(mX, mY)=0$, as well $lN_{F}(lX, lY)=N_{F}(lX, lY)$. i.e. $l$ and $m$  are both 
integrable.
\end{proof}

By virtue of Proposition \ref{pro_1.1}, Theorem \ref{th_2} and Theorem \ref{th_9}, we get the following Theorem. 

\begin{theorem}\label{th_10.1} Given an $F$-structure manifold $M$, the following conditions are equivalent
\begin{eqnarray*}
(i) &&l \; \textit{and} \; m \; \textit{are integrable}\\	
(ii) && D_{l}\;\textit{and}\; D_{m}\;\textit{are integrable}\\
(iii) && N_{l}(X, Y)=0,\\
(iv) && N_{F}(X, Y)=lN_{F}(lX, lY)+N_{F}(lX, mY)+N_{F}(mX, lY),
\end{eqnarray*}
for any vector fields $X$ and $Y$  on $M$.
\end{theorem}

\section{Partial integrability and complete integrability of $F$-structure}\label{Sec6}
Suppose that the distribution $D_{l}$ is integrable and take an arbitrary vector field $U$ in an integral manifold of $D_{l}$. We define an operator $\widetilde{F}$ by 
\begin{eqnarray*}
\widetilde{F}U = FU,
\end{eqnarray*}	 
then $\widetilde{F}$ leaves invariant tangent spaces of every integral manifolds of $D_{l}$. Also,  
\begin{eqnarray*}
(\alpha\widetilde{F}^{K}+ \beta \widetilde{F}^{K-1})^{1/2},
\end{eqnarray*}	
acts as an almost product structure on each integral manifold of $D_{l}$. 

For any vector fields $U$ and $V$ tangent to integral manifold of $D_{l}$, we denote by 
\begin{eqnarray*}
N_{\widetilde{F}}(U ,V)=[\widetilde{F}U, \widetilde{F}V] - \widetilde{F}[\widetilde{F}U, V] -\widetilde{F}[U, \widetilde{F}V] + \widetilde{F}^{2}[U, V],
\end{eqnarray*}
the Nijenhuis tensor of the structure $\widetilde{F}$ induced on each integral manifold of $D_{l}$ from the structure $F$. Then we have 
\begin{eqnarray}\label{eq_VV}
N_{\widetilde{F}}(lX ,lY)=N_{F}(lX ,lY),
\end{eqnarray}
for any vector fields $X$ and $Y$ on $M$. Indeed since the distribution $D_{l}$  is integrable, we find
\begin{eqnarray*}
N_{\widetilde{F}}(lX ,lY)&=&[\widetilde{F}lX, \widetilde{F}lY] - \widetilde{F}[\widetilde{F}lX, lY] -\widetilde{F}[lX, \widetilde{F}lY] + \widetilde{F}^{2}[lX, lY]\\
&=&[FlX, FlY] - \widetilde{F}[FlX, lY] -\widetilde{F}[lX, FlY] + F^{2}[lX, lY]\\
&=&[FlX, FlY] - F[lFX, lY] -F[lX, lFY] + F^{2}[lX, lY]\\
&=&N_{F}(lX ,lY).
\end{eqnarray*}

\begin{definition}\cite{Y.K}\label{de_01}
We call an $F$-structure to be partially integrable if the distribution $D_{l}$ is integrable 
and the structure $\widetilde{F}$ induced from $F$ on each integral manifold of $D_{l}$ is integrable. see\cite{D.R,S.V}.
\end{definition}

\begin{theorem}\label{th_10} Given an $F$-structure manifold $M$. A necessary and sufficient condition for an $F$-structure to be partially integrable is that one of the following equivalent conditions be satisfied:  
\begin{eqnarray}\label{eq_WW}
N_{F}(lX ,lY)=0,
\end{eqnarray}
or
\begin{eqnarray*}
N_{F}(FX, FY)=0,
\end{eqnarray*}
for any vector fields $X$ and $Y$  on $M$.
\end{theorem} 

\begin{proof}
Suppose that $F$-structure is partially integrable, then from \eqref{eq_12.1} and \eqref{eq_VV}, we find $N_{\widetilde{F}}(lX ,lY)=0\Leftrightarrow N_{F}(lX,lY)=0 \Leftrightarrow N_{F}(FX,FY)=0$.\\ 
Conversely, from \eqref{eq_12.1}, we have $N_{F}(lX ,lY)=0 \Leftrightarrow N_{F}(FX ,FY)=0$, then the, by \eqref{eq_VV}, the structure $\widetilde{F}$ is integrable. Also $N_{F}(lX ,lY)=0$,  implies   $mN_{F}(lX ,lY)=0$, by Theorem \ref{th_5}, we find, $D_{l}$ is integrable. Thus $F$-struoture is partially integrable.	
\end{proof}

\begin{definition}\cite{Aqe}\label{de_02}
Given an $F$-structure manifold $M$. An $F$-structure is said to be completely integrable if the distribution $D_{l}$ and $D_{m}$ are both integrable, and the structure $\widetilde{F}$ induced from $F$ on each integral manifold of $D_{l}$ is integrable. .
\end{definition}

From Definition \ref{de_01} and Definition \ref{de_02}, we have the following theorem. 

\begin{theorem}\label{th_11} Given an $F$-structure manifold $M$. A necessary and sufficient condition for an f-structure to be completely integrable is that the distribution $D_{m}$ is integrable and that the $F$-structure  is partially integrable.
\end{theorem} 

\begin{theorem}\label{th_12} Given an $F$-structure manifold $M$. In order that the $F$-structure to be completely integrable, it is necessary and sufficient that
\begin{eqnarray}\label{eq_YY}
N_{F}(X ,Y)=N_{F}(lX ,mY)+ N_{F}(mX, lY).
\end{eqnarray}
for any vector fields $X$ and $Y$ on $M$.
\end{theorem} 

\begin{proof}
$i)$ Suppose that the $F$-structure is a completely integrable, i.e. $D_{m}$ is integrable and $F$-structure is partially integrable. Using \eqref{eq_1.7}, \eqref{eq_2.1} and \eqref{eq_WW}, we get \eqref{eq_YY}.\\
$ii)$ Conversely, assume that \eqref{eq_YY} is satisfied. Using \eqref{eq_2.1}, we find
\begin{eqnarray*}
N_{F}(lX, lY)+N_{F}(mX, mY)=0.
\end{eqnarray*}
In this relation we replace $X, Y$ with $mX, mY$ respectively, we get \eqref{eq_1.7}, as well \eqref{eq_WW}, i.e. $D_{m}$ is integrable and $F$-structure is partially integrable, hence the $F$-structureis is completely integrable.
\end{proof}

\begin{theorem}\label{th_13}Given an $F$-structure manifold $M$. In order that the $F$-structure to be integrable, it is necessary and sufficient that the $F$-structure is completely integrable and
\begin{eqnarray*}
N_{F}(lX ,mY)=- N_{F}(mX, lY).
\end{eqnarray*}
for any vector fields $X$ and $Y$  on $M$.
\end{theorem}

\section{Special cases}\label{Sec7}
\subsection*{Case.1} If $\alpha=0$, $\beta=1$ and $K=3$. The $F$-structure satisfying $\eqref{eq_A}$ becomes of the form 
\begin{eqnarray*}
F^{3}+F=0,	
\end{eqnarray*}	
which was studied in \cite{D.R,I.Y,Yan1,Yan2,Y.K}.

\subsection*{Case.2} If $\alpha=0$, $\beta=-1$ and $K=3$. The $F$-structure satisfying $\eqref{eq_A}$ becomes of the form 
\begin{eqnarray*}
F^{3}-F=0,	
\end{eqnarray*}	
which was studied in \cite{D.R,Mat,Pok,S.V}.

\subsection*{Case.3} If $\alpha=0$, $\beta=\dfrac{1}{\lambda^{2}}$ and $K=3$. The $F$-structure satisfying $\eqref{eq_A}$ becomes of the form 
\begin{eqnarray*}
F^{3}+\lambda^{2}F=0,	
\end{eqnarray*}	
which was studied in \cite{U.G}.

\subsection*{Case.4} If $\alpha=0$, $\beta=1$ and $K=5$. The $F$-structure satisfying $\eqref{eq_A}$ becomes of the form 
\begin{eqnarray*}
F^{5}+F=0,	
\end{eqnarray*}	
which was studied in \cite{And1,And2,Sin2}.

\subsection*{Case.5} If $\alpha=0$, $\beta=-1$ and $K=5$. The $F$-structure satisfying $\eqref{eq_A}$ becomes of the form 
\begin{eqnarray*}
F^{5}-F=0,	
\end{eqnarray*}	
which was studied in \cite{Aqe,Pok}.

\subsection*{Case.6} If $\alpha=1$ and $\beta=0$. The $F$-structure satisfying $\eqref{eq_A}$ becomes of the form 
\begin{eqnarray*}
F^{K+1}+F=0,	
\end{eqnarray*}	
which was studied in \cite{Gup}.

\subsection*{Case.7} If $\alpha=0$ and $\beta=(-1)^{K+1}$. The $F$-structure satisfying $\eqref{eq_A}$ becomes of the form 
\begin{eqnarray*}
F^{K}+(-1)^{K+1}F=0,	
\end{eqnarray*}	
which was studied in \cite{Das1,Das2}.

\subsection*{Case.8} If $\alpha=0$, $\beta=1$ and $K=p_{1}p_{2}$. The $F$-structure satisfying $\eqref{eq_A}$ becomes of the form 
\begin{eqnarray*}
F^{p_{1}p_{2}}+F=0,	
\end{eqnarray*}	
which was studied in \cite{S.G1}.

\subsection*{Case.9} If $\alpha=0$, $\beta=1$ and $K=p^{2}+2$. The $F$-structure satisfying $\eqref{eq_A}$ becomes of the form 
\begin{eqnarray*}
F^{p^{2}+2}+F=0,	
\end{eqnarray*}	
which was studied in \cite{S.G2}.

\section{Examples}\label{Sec8}
\begin{example}\label{ex_1}
On $M= \{(x, y)\in\mathbb{R}^{2},\; y\neq0\}$ ($2$-dimensional manifold), we define the tensor  $F$ of type $(1,1)$, by
$$F=\left(\begin{array}{ccc}
-1& y \\
\dfrac{-1}{y}& 2
\end{array}\right).$$
It is easy to find out that  $rank(F)=2$ and for $\alpha=1,\beta=-2, K=3$, we find 
$$F^{4} -2F^{3} +F=0,$$
$$l=-F^{3}+2F^{2}=I,\quad m=I-l=0,$$
$$(D_{l})_{(x, y)}= T_{(x, y)}M,\quad (D_{m})_{(x, y)}= \left\{0\right\}.$$
It's easily verified that 
\begin{eqnarray*}
N_{F}(\partial_{x}, \partial_{y})=0,
\end{eqnarray*}
i.e. $F$ is integrable  (partially, completely) integrable, as well $D_{l}$ and  $D_{m}$ are integrable.
\end{example}

\begin{example}\label{ex_2}
Let $(x_{1}, x_{2})$ be the Cartesian coordinates of $\mathbb{R}^{2}$ and $F$ be a tensor of type $(1,1)$, defined by
$$F=\left(\begin{array}{ccc}
1& -1 \\
1& -2
\end{array}\right).$$
It is easy to find out that  $rank(F)=2$ and for $\alpha=-1,\beta=-2, K=3$, we find 
$$-F^{4} -2F^{3} +F=0,$$
$$l=F^{3}+2F^{2}=I,\quad m=I-l=0,$$
$$(D_{l})_{x}= T_{x}\mathbb{R}^{2},\quad (D_{m})_{x}= \left\{0\right\}.$$
where $x=(x_{1}, x_{2})\in \mathbb{R}^{2}$, it's easily verified that 
\begin{eqnarray*}
N_{F}(\partial_{x_{i}}, \partial_{x_{j}})= [F\partial_{x_{i}}, F\partial_{x_{j}}] - F[F\partial_{x_{i}}, \partial_{x_{j}}] -F[\partial_{x_{i}}, F\partial_{x_{j}}] + F^{2}[\partial_{x_{i}}, \partial_{x_{j}}]=0,
\end{eqnarray*}
for all $i,j = 1, 2 $, i.e. $F$ is integrable  (partially, completely) integrable, as well $D_{l}$ and  $D_{m}$ are integrable.\\
We put $\widehat{F}^{2}=-F^{3} -F^{2}=-I$, i.e. $\widehat{F}=(-I)^{1/2}=\left(\begin{array}{ccc}
a& b \\
\dfrac{1+a^{2}}{b}& -a
\end{array}\right)$, where $a, b$ are real constants and $b\neq0$. Because $\widehat{F}$ is integrable, then 
\begin{eqnarray*}
H&=&\left\lbrace X-j\widehat{F}X,\; X\in D_{l}\right\rbrace=\left\lbrace X-j\left(\begin{array}{ccc}
a& b \\
\dfrac{1+a^{2}}{b}& -a
\end{array}\right)X,\; X\in T\mathbb{R}^{2}\right\rbrace\\
&=&	\left\lbrace \left(\begin{array}{cc}
x-j(ax+by) \\
y+j(\dfrac{1+a^{2}}{b}x+ay)
\end{array}\right),\; x, y \in \mathbb{R}\right\rbrace
\end{eqnarray*}	
is $CR$-structure.
\end{example}

\begin{example}\label{ex_3}
On $M= \{x=(x, y, z)\in\mathbb{R}^{3},\; x\neq0\}$ ($3$-dimensional manifold), we define the tensor  $F$ of type $(1,1)$, by
$$F=\left(\begin{array}{cccc}
0& 0& x\\
0& 0&0\\
\dfrac{-1}{x}&0&-1
\end{array}\right).$$
It is easy to find out that  $rank(F)=2$ and for $\alpha=\beta=1, K=5$, we find 
$$F^{6} +F^{5} +F=0,$$
$$l=-F^{5}-F^{4}=\left(\begin{array}{cccc}
1& 0& 0\\
0& 0&0\\
0&0&1
\end{array}\right),\quad m=I-l=\left(\begin{array}{cccc}
0& 0& 0\\
0& 1&0\\
0&0&0
\end{array}\right),$$
$$(D_{l})_{(x, y, z)}= Span\left\{\partial_{x}, \partial_{z} \right\},\quad (D_{m})_{(x, y, z)}= Span\left\{\partial_{y}\right\}.$$
It's easily verified that 
\begin{eqnarray*}
N_{F}(\partial_{x}, \partial_{y})=N_{F}(\partial_{x}, \partial_{z})=N_{F}(\partial_{y}, \partial_{z})=0,
\end{eqnarray*}
i.e. $F$ is integrable  (partially, completely) integrable, as well $D_{l}$ and  $D_{m}$ are integrable.
\end{example}

\begin{example}\label{ex_4}
On $M= \{(x, y, z, t)\in\mathbb{R}^{4},\; x\neq0\}$ ($4$-dimensional manifold), we define the tensor  $F$ of type $(1,1)$, by
$$F=\left(\begin{array}{cccc}
1& 0& 0& -x\\
0& 1&\dfrac{-1}{x} & 0\\
0& x& 0& 0\\
\dfrac{1}{x}& 0& 0& 0
\end{array}\right).$$
It is easy to find out that  $rank(F)=4$ and for $\alpha=-1,\beta=1, K=5$, we find
$$-F^{6} +F^{5}+F=0,$$ 
$$l=-(-F^{5} +F^{4})=I,\quad m=I-l=0,$$
$$(D_{l})_{(x, y, z, t)}= T_{(x, y, z, t)}M,\quad (D_{m})_{(x, y, z, t)}= \left\{0\right\}.$$
For all vector fields $X$ and $Y$  on $M$, we have, $mN_{F}(X,Y)=lN_{F}(mX,mY)=0$, i.e. $D_{l}$ and $D_{m}$ are both integrable.
\begin{eqnarray*}
N_{F}(\partial_{z}, \partial_{t})&=& [F\partial_{z}, F\partial_{t}] - F[F\partial_{z}, \partial_{t}] -F[\partial_{z}, F\partial_{t}] + F^{2}[\partial_{z}, \partial_{t}]\\
&=&\left[ \dfrac{-1}{x}\partial_{y}, -x\partial_{x}\right]  - F\left[ \dfrac{-1}{x}\partial_{y}, \partial_{t}\right]  -F\left[ \partial_{z}, -x\partial_{x}\right]  + 0\\
&=&-\dfrac{1}{x}\partial_{y}\neq 0,
\end{eqnarray*}
hence $F$ is not integrable. On the other hand, we have
$$N_{F}(l\partial_{z},l\partial_{t})=N_{F}(\partial_{z},\partial_{t})\neq 0,$$ 
then $F$ is not partially (completely) integrable.
\end{example}


\end{document}